\documentclass[11pt]{amsart}
\usepackage{amssymb, latexsym, mathrsfs, color, tikz, multirow,mathtools,xcolor,enumerate,caption, subcaption, graphicx,xcolor,comment,multicol}
\usepackage{graphics}
\graphicspath{ {./Diagrams/} }

\usepackage{amsmath,amssymb,amsthm}
\usepackage{xcolor,caption}

\usetikzlibrary{shapes,snakes}

\usepackage{ytableau} %Young Tableau

\numberwithin{equation}{section}
\theoremstyle{plain}
\newtheorem{theorem}[equation]{Theorem}
\newtheorem{proposition}[equation]{Proposition}
\newtheorem{cor}[equation]{Corollary}
\newtheorem{lemma}[equation]{Lemma}
\newtheorem{con}[equation]{Conjecture}
\newtheorem {problem}[equation]{Problem}
\newtheorem {question}[equation]{Question}

\newtheorem{innercustomthm}{Theorem}
\newenvironment{customthm}[1]
  {\renewcommand\theinnercustomthm{#1}\innercustomthm}
  {\endinnercustomthm}

\theoremstyle{definition}
\newtheorem{definition}[equation]{Definition}
\newtheorem{example}[equation]{Example}
\newtheorem{remark}[equation]{Remark}

\def\N{\mathbb N}

\keywords{dice relabeling, cyclotomic polynomials, generating functions}


\title{Dice Relabeling using Square-Sided Dice}
\author{Evelyn Fiore, George D. Nasr, Cooper Stone}
\date{}

\address{Department of Mathematics, Augustana University} \email{george.nasr@augie.edu} \email{ljferguson22@ole.augie.edu} \email{castone23@ole.augie.edu}
\begin{document}
\maketitle

\begin{abstract}
We continue recent work of Chao, Gabel, Larson, and Nasr in using cyclotomic polynomials for dice relabeling. In their work, one idea they expand on is finding pairs of dice with different number of sides which maintain the sum frequency of two normal dice. We continue this idea in this paper by studying pairs of dice where the number of sides of each is a different perfect square (which we call ``square-sided" dice). We additionally provide conjectures offering ideas for future exploration.\end{abstract}

\section{Introduction and Summary of Results}

George Sicherman posed and solved the following question.

\begin{question}
 How many ways can one label two six-sided dice so that the frequency of all possible sums remain the same as if they were both labeled $1$ through $6$?
\end{question}

Sicherman found that the answer was two. Either one can use the usual labeling on both dice (often called the ``standard" solution), or one can label one dice $1,2,2,3,3,4$ and the other dice $1,3,4,5,6,8$ (often called the ``Sicherman" dice). This result was discussed further and reported by Martin Gardner \cite{gardner}. Inspired by this work, Broline explored this question for an arbitrary number of platonic solids \cite{b}. Gallian and Rusin addressed the more general question \cite{gc}.

\begin{question}\label{qu:gen}
  Given $n$ dice, each with labels $1$ through $m$, how many way can these dice be relabeled without altering the frequencies of the sums? 
\end{question}
 By encoding the data of the frequencies and the labels on the dice as a generating function, Broline, Gallian, and Rusin observed that one could factor the generating function encoding the frequencies using cyclotomic polynomials (a method we go into further detail on later). Using this technique, Gallian and Rusin were able to demonstrate that for any number of dice, there are three possible dice that could be used to answer Question \ref{qu:gen} if $m$, the number of sides, is a product of two (not necessarily distinct) prime numbers \cite[Theorem 2]{gc}. They additionally have results for when $m$ is a prime power and many other related questions to relabeling dice. 

Many different types of results involving dice relabeling followed. In \cite{sm}, they enumerate the frequency of a particular sum given $n$ $m$-sided dice. In \cite{fs}, they characterize the numbers that can be realized as the sums of relabeled six-sided dice. Other papers explored changing the probabilities of the sums from the usual one given by $n$ $m$-sided dice. For instance, authors of \cite{bmrs,lr,bs,m} explored different questions assuming ``equally likely sums", that is, all sums are equally likely. In \cite{rss}, they consider ``Pythagorean dice" which provide an alternative probability distribution on the possible sums. 

However, there remains many generalizations of Gallian's and Rusin's results involving the case where we use the original probabilities. Indeed, at the end of their paper, Gallian and Rusin leave the readers with two different further explorations of their ideas.
\begin{question}\label{qu:three}
How many relabeling are there in the case where $m=p^2q$ or $m=pqr$ (where $p,q,r$ are distinct primes)? 
\end{question}
\begin{question} \label{qu:dif_size}
Can one find dice, not necessarily with the same number of sides, matching the frequencies of $n$ $m$-sided dice? 
\end{question}

In \cite{firstpaper}, the authors followed the techniques of Gallian and Rusin to provide partial answers to both questions in the case of $n=2$. However, much of their efforts were targeted towards the first question. In this paper we wish to expand further on the second. Specifically, we ask the following special case of Question \ref{qu:dif_size}.

\begin{question} \label{qu:dif_size_square}
Can one find dice, each having a perfect square as its number of sides, not necessarily with the same number of sides, matching the frequencies of $n$ $m$-sided dice? 
\end{question}
\noindent Throughout this paper, we refer to dice whose number of sides is a perfect square as a \textit{square-sided} dice. 

Our paper is organized as follows. In Section \ref{sec:term}, we go over terminology for this paper. In Section \ref{sec:cyc}, we go over how we use generating functions and cyclotomic polynomials to reframe the above questions. In that section, we also include cyclotomic polynomial identities. Our new results start in Section \ref{sec:form}, where we ultimately prove the following.

\begin{customthm}{\ref{theorem:cs2n}}
Let $a,b,m\in \N$ so that $m=ab$. There is an $a^2$- and $b^2$-sided dice which have the same frequency of sums as two $m$-sided dice. 
\end{customthm}
\noindent This result only offers one labeling for such dice---see Corollary \ref{cor:ab} for the explicit description. The trade off here, though, is there is no restriction on $m$.

In contrast, Sections \ref{sec:psquaredq} and Section \ref{sec:pqr}, we add requirements on $m$ by combining Questions \ref{qu:three} and \ref{qu:dif_size_square}, separately dealing with the two cases. In the case of $m=p^2q$, we have nearly completely classified one way of choosing labels for these dice.
\begin{customthm}{\ref{thm:positive}}
Let $p$ and $q$ be distinct prime numbers. There are $6$ pairs of dice, one with $p^4$ sides and the other with $q^2$ sides, which have the same frequency sums as two $p^2q$-sided dice. 
\end{customthm}

\noindent We further are able to provide one more pair of dice in the case where $p=2$ and $q\equiv 1\mod 4$ in Theorem \ref{thm:q01_coeffs}, claiming all other cases lead to negative coefficients (see Conjecture \ref{con:q01}). Despite this, note that one could also conceivably find a pair of $p^2q^2$- and $p^2$-sided dice which could also have the same frequency sum. This is precisely Problem \ref{prob:open}, which we offer to the reader with preliminary thoughts as one potential source of a future project. 

In contrast to our work in Section \ref{sec:psquaredq} for $m=p^2q$, our work in Section \ref{sec:pqr} for $m=pqr$ case is much more preliminary. We state the following.
\begin{customthm}{\ref{thm:pqr}}
Let $p$, $q$, and $r$ be distinct prime numbers. There are $11$ pairs of dice, one with $p^2q^2$ sides and the other with $r^2$ sides, which have the same frequency sums as two $pqr$-sided dice.
\end{customthm}

\noindent As in Section \ref{sec:psquaredq}, we provide one more pair of dice when $p=2$ and $q<r$ in Lemma \ref{lem:pqr}. However, in this section, we leave $33$ open cases conjectured to have negative coefficients (see Conjecture \ref{con:pqr}).

In all our results, note there is a trade-off between finding multiple pairs of square-sided dice versus being able to explicitly classify the labels on the new dice.

\subsubsection*{Acknowledgments}

We thank Augustana Research and Artist Fund for providing initial support funds this project. We also thank Drew Alton and the South Dakota Space Grant Consortium for further funding the work. 

\section{Terminology and Notation}\label{sec:term}

We follow the primary terminology set up by \cite{gc}, namely: 
\begin{definition}\leavevmode
\begin{itemize}
\item The dice labeled $1$ through $m$ is called a \textbf{standard dice}. 
\item A dice with $m$ sides has \textbf{size $m$}.
\end{itemize}
\end{definition}

\noindent We further add the following. 

\begin{definition}\label{def:nms}
Given a set of $n$ dice which have the same frequencies of sums of $n$ standard $m$-sided dice, any one of these dice is called an \textbf{$(n,m)$-solution}. 
\end{definition}

%Thus, Gallian and Rusin showed that when $m$ is a product of two (not necessarily distinct) prime numbers, there are three solutions. When $n=2$, this gives rise to two possible pairs: two standard dice and the two other dice guaranteed by their result. 

\begin{remark}\leavevmode
\begin{enumerate}
\item 
Technically, when \cite{gc} defines ``solution", they require the corresponding dice have size $m$. We remove this condition as results in this paper never have dice with the same size. 
\item We add $n$ and $m$ to the definition of solution in Definition \ref{def:nms} because values will change throughout this paper and this allows conciseness when discussing what our new dice accomplish. That said, note that, as indicated in the introduction, $n$ will only be $2$ throughout this paper. We offer the more general notation as a suggestion to be used in further work.
\end{enumerate}
\end{remark}

Finally, a remark on notation. At times we will want to write out our formulations for sides of a dice. When working with an arbitrary number of sides, we will use $x^{(n)}$ to represent $x$ listed out $n$ times. For instance, we may use $3^{(4)}=3,3,3,3$ to indicate the dice has four sides labeled with $3$.

\begin{example}
The Sicherman dice have labels $1,2^{(2)},3^{(2)},4$ and $1,3,4,5,6,8$.
\end{example}

\section{Using Generating Functions and Cyclotomic Polynomials} \label{sec:cyc}

Question \ref{qu:gen} can be reframed in the following way: How many collections of polynomials $P_1,P_2,\dots, P_n$, all with non-negative integer coefficients, are there so that  $P_i(1)=m$ for all $i$ and
\[P_1P_2\cdots P_n=\left( \sum_{i=1}^m x^i \right)^n=x^n\left( {x^m-1\over x-1}\right)^n?\]

Here, each polynomial $P_i$ is the generating function of a given dice. That is, $[x^j]P_i$, the $j$th coefficient of $P_i$, is the number of sides of the $i$th dice labeled $j$. Thus, $\displaystyle \sum_{i=1}^m x^i$ is the generating function for the standard dice of size $m$. Thus, if we denote the product of the $P_i$'s as $F(x)$, then $[x^j]F$ is the number of ways which a sum of $j$ can be achieved from the $n$ dice corresponding to the polynomials $P_1,P_2,\dots, P_n$. We refer to $F$ as the \textit{frequency polynomial}. Throughout this paper, we will see different formulations of $F$ depending on the size of the dice involved. Regardless of the explicit formulation of $F$, observe this gives a bijection between collections of dice answering our question and factorizations of $F$ with $P_i(1)=m$ and $[x^j]P_i\geq 0$ for all $i$ and $j$. To this end, we also refer to each $P_i$ as \textit{solution}, keeping in mind we really are referring to the dice that $P_i$ represents. 

\begin{remark}
Note that Question \ref{qu:dif_size} only differs from the above description in the evalutions of $P_i$, in that we may have distinct values for when we compute $P_i(1)$. This will be made explicit in the forthcoming sections.
\end{remark}

Observe that addressing the above factorization problem relies on understanding the factorizations of $x^m-1$. The factors of this polynomial is completely understood, as its irreducible factors are the \textit{cyclotomic polynomials}. Results on these polynomials are well understood and we will survey some of what is known here, starting with the definitions.

\begin{definition}
Let $m$ be a non-negative integer. 
\begin{itemize}
\item A (complex-valued) solution to $x^m=1$ is called an \textit{$m$th root of unity}.
\item A \textit{primitive $m$th root of unity} is an $m$th root of unity which is not a $j$th root of unity for any $j<m$.
\item The \textit{cyclotomic polynomial} $\phi_m(x)$ is the polynomial whose roots are the primitive $m$th roots of unity.
\end{itemize}
\end{definition}

The following known identities or readily computed by identities that we will use throughout our paper. 
%In what follows, $\mu(n)$ is the \emph{M\"obius function}, which is $1$ when $n$ is square-free with an even number of prime numbers, $-1$ when $n$ is square-free with an odd number of prime numbers, and $0$ otherwise. Additionally, throughout, we assume $p,q,r$ are a prime numbers. 

\begin{align}
x^n-1 &=\prod_{d|n}\phi_d(x)\label{eq:prod_cyc}\\
\phi_n(x) &=\prod_{d|n}(x^d-1)^{\mu(n/d)}\label{eq:prod_roots}\\
\phi_p(x)&=\sum_{i=0}^{p-1}x^i={x^p-1\over x-1}\label{eq:cyc_prime}\\
\phi_{p^km}(x)&=\phi_{pm}\left(x^{p^{k-1}}\right), \text{where $m$ is $1$ or relatively prime to $p$.}\label{eq:cyc_prime_power_mult}\\
\phi_m(x)\phi_{pm}(x) &=\phi_m(x^p), \text{where prime $p$ does not divide $m$.}\label{eq:divisor} \\
\phi_{n}(1)&=\begin{cases}p & n=p^k\\1 & \text{otherwise} \end{cases}\label{eq:eval_prime_power}\\
\phi_{p^kq}(x) &= \frac{(x^{p^{k-1}}-1)(x^{p^kq}-1)} {(x^{p^k}-1)(x^{p^{k-1}q}-1)}\\
\phi_{pqr}(x)&=\frac{(x^p-1)(x^q-1)(x^r-1)(x^{p q r}-1)}{(x-1)(x^{p q}-1)(x^{pr}-1)(x^{q r}-1)}\label{eq:pqr}\\
\phi_{2n}(x)&=\phi_n(-x)\text{, where $n$ is an odd integer greater than 1}
\label{eq:parity}
\end{align}

\begin{remark}\leavevmode
\begin{enumerate}
\item 
One should note that the factors of the form $x^i-1$ could all be written as $1-x^i$ and the results, as far as this paper is concerned, will remain unchanged. In future sections, we will often write rational functions in both ways, depending on whatever is convenient for our proofs.
\item We will often omit the input to cyclotomic polynomials, writing $\phi_p$ instead of $\phi_p(x)$, for instance.  This allows us to write more complicated expressions while permitting ease of reading.
\end{enumerate}
\end{remark}

We will state one additional result from \cite{firstpaper} which is useful for results in this paper.

%\begin{lemma}
%If $p$ and $q$ are prime numbers, then for all integers $k\geq 0$,
%\[ \prod_{i=0}^k \phi_{p^iq}(x)=\phi_q(x^{p^k})\]\label{lem:product_piq}
%\end{lemma}
%

\begin{lemma}\label{thm:josh_pqr}
For three distinct prime numbers, $p,q,r$,
$$\phi_p(x)\cdot\phi_{pq}(x)\cdot\phi_{pr}(x)\cdot\phi_{pqr}(x)=\phi_p(x^{qr}).$$
\end{lemma}

Before proceeding to our main arguments and results, we discuss one final technique that we will take extensive advantage of. As we have seen, cyclotomic polynomials can be expressed as rational functions. We will often consider product of these cyclotomic polynomials, and because of the similarities in the different formulas for the cases we will be considering, a lot of cancellation occurs between common factors. For example, for primes $p$ and $q$, we have 
\[\phi_p\phi_{pq}={x^p-1\over x-1}{(x-1)(x^{pq}-1)\over (x^p-1)(x^q-1)}={x^{pq}-1\over x^q-1}.\]
We can further ``simplify" this by writing it as a product of (possibly finite) formal series:
\[{x^{pq}-1\over x^q-1}={1-x^{pq}\over 1-x^q}=(1-x^{pq})\sum_{i=0}^\infty x^{iq}.\]
Writing our functions in this way will be extremely useful in demonstrating when certain options for solutions yield negative coefficients, and thus are not solutions.

\section{A method for forming two square-sided dice}\label{sec:form}

In this section, we demonstrated a method for finding two dice whose sizes are perfect squares with the same frequency of sums for the two standard dice of any fixed size.

\begin{theorem}\label{theorem:cs2n}
Let $m,a,b\in \N$ so that $m = ab$. There is an $a^2$- and $b^2$-sided dice which form an $(2,m)$-solution.
\end{theorem}

\begin{proof}
%% should we specifiy that n^2 is equal to (ab)^2
We define two functions, $A(x)$ and $B(x)$, to be the generating functions for the $a^2$ and $b^2$ sided dice: 
$$A(x)=x\frac{(1-x^a)^2}{(1-x)^2}$$ %%I feel like I need to show where this came from, but I'm struggling to figure that out
and 
$$B(x)=x\frac{(1-x^{ab})^2}{(1-x^a)^2}.$$
%$A(x)$ is the $a^2$ model, and by Proposition \ref{prop:poly} becomes

Rewriting $A(x)$ in polynomial form, we have 
$$A(x)=ax^a+\sum_{k=1}^{a-1}k(x^k+x^{2a-k}),$$
and so 
$$A(1)=a+\sum_{k=1}^{a-1}2k=a+a(a-1)=a^2,$$
where the sum simplifies due to Proposition \ref{prop:triangle}.
Now, note $B(x)$ can be rewritten as
$$B(x)=x\frac{(1-(x^a)^b)^2}{(1-(x^a))^2}$$
and so when $B(x)$ is written in polynomial form, we see we will also have $B(1)=b^2$. 

Finally, note, 
\[A(x)B(x)=x^2{(1-x^{ab})^2\over (1-x)^2},\]
and since $ab=m$, these function are generating functions for a pair of $a^2$- and $b^2$-sided dice giving the same frequencies of sums of two $n$-sided dice, as desired.

\end{proof}

By interpreting the coefficients of the polynomials of $A(x)$ and $B(x)$ in the prior result, one gets the following. 
\begin{cor}\label{cor:ab}

One can form an $a^2$-sided dice with labels
\[1,2^{(2)}, 3^{(3)},\cdots, (a-1)^{(a-1)}, a^{(a)}, (a+1)^{(a-1)}\cdots, (2a-2)^{(2)},2a-1\]
and a $b^2$-sided dice with labels {\small
\[ 1,(a+1)^{(2)},(2a+1)^{(3)},(3a+1)^{(4)},\cdots,(n-b-a+1)^{(b-1)},(n-a+1)^{(b)},(n-a+b+1)^{(b-1)},\cdots, (2n-3a+1)^{(2)}, 2n-2a+1\]}
so that these dice are $(2,m)$-solutions, where $m-ab$.
\end{cor}
\begin{proof}
The proof of the prior result already includes the generating function for $A(x)$ written in polynomial form. Note 
\[B(x)=bx^{m-a+1}+\sum_{k=1}^{b-1}k(x^{ak-a+1}+x^{2m-ak-a+1}).\]
\end{proof}

Interestingly, note that assuming $a,b>1$, swapping the roles of $a$ and $b$ yields different labels. We demonstrate this in the next example.

\begin{example} 

Let $n=6=2\times 3$. If we let $a=2$ and $b=3$, our generating functions from the prior result yield a $4$- and $9$-sided dice as given in the following table. 
%\begin{center}

%A-typical Dice Distribution Example
%\begin{tabular}{l|llllllllllll}
%  & $1$ & $2$ & $3$ & $4$ & $4$ & $5$ & $5$ & $6 $ & $6$ & $7$  & $8$  & $9$  \\ \hline
%$1$ & $2$ & $3$ & $4$ & $5$ & $5$ & $6$ & $6$ & $7$ & $7$ & $8 $ &$ 9$  & $10$ \\
%$2$ & $3$ & $4$ & $5$ & $6$ & $6$ & $7$ & $7$ & $8$ & $8$ & $9$  & $10$ & $11$ \\
%$3$ & $4$ & $5$ & $6$ & $7$ & $7$ & $8$ & $8$ & $9$ & $9$ & $10$ & $11 $& $12$
%\end{tabular}
%\end{center}
%\end{example}

%\begin{example}
\begin{center}

\begin{tabular}{c|cccccccccc}
  & $1$ & $3$ & $3$ & $5$ & $5$ & $5$ & $7$ & $7 $ & $9$  \\ \hline
$1$ & $2$ & $4$ & $4$ & $6$ & $6$ & $6$ & $8$ & $8$ & $10$ & \\
$2$ & $3$ & $5$ & $5$ & $7$ & $7$ & $7$ & $9$ & $9$ & $11$ \\
$2$ & $3$ & $5$ & $5$ & $7$ & $7$ & $7$ & $9$ & $9$ & $11$ \\
$3$ & $4$ & $6$ & $6$ & $8$ & $8$ & $8$ & $10$ & $10$ & $12$
\end{tabular}
\end{center}

If instead we let $a=3$ and $b=2$, note interestingly that while we still get a $4$- and $9$-sided dice, the prior result offers alternative dice labels.
\begin{center}

\begin{tabular}{c|ccccccccc}
  & $1$ & $2$ & $2$ & $3$ & $3$ & $3$ & $4$ & $4$ & $5$ \\ \hline
  
$1$ & $2$ & $3$ & $3$ & $4$ & $4$ & $4$ & $5$ & $5$ & $6$ \\

$4$ & $5$& $6$ & $6$ & $7$ & $7$ & $7$ & $8$ & $8$ & $9$ \\

$4$ & $5$ & $6$& $6$ & $7$ & $7$& $7$ & $8$& $8$ & $9$\\

$7$ & $8$& $9$& $9$ & $10$ & $10$ & $10$ & $11$ & $11$ & $12$
\end{tabular}

\end{center}

Observe in both examples, we get the same frequencies of sums as two standard $6$-sided dice. That is, these dice are indeed $(2,6)$-solutions.
\end{example}

\section{Square-sided dice for pairs of $p^2q$-sided dice}\label{sec:psquaredq}

In this section, we explore $(2,p^2q)$-solutions. Note that the generating function for the frequency sum would be
\[x^2{(1-x^{p^2q})^2\over (1-x)^2}=x^2\phi_p^2\phi_q^2\phi_{p^2}^2\phi_{pq}^2\phi_{p^2q}^2\]
by Equation \eqref{eq:prod_cyc}.

There are two different pairs of square-sided dice we could find for them to be a $(2,p^2q)$-solution: either the dice have sizes $p^4$ and $q^2$ or they have sizes $p^2q^2$ and $p^2$. We proceed with the former, leaving the later as an open case for the reader to explore: see Problem \ref{prob:open}. We start with the following definition.

\begin{definition}
%% Should we change a and b in Q to alpha and beta?
For $a,b\in \{0,1,2\}$, we let
$$P_{a,b}(x):=x{\phi}_{p^2}^2{\phi}_{p}^2{\phi}_{pq}^a{\phi}_{p^2q}^b$$
and
$$Q_{a,b}(x):=x{\phi}_{q}^2{\phi}_{pq}^a{\phi}_{p^2q}^b.$$

\end{definition}

\begin{remark}

Because the values for $a$ and $b$ are single-digit numbers, we often omit the comma from the notation for these polynomials, writing $P_{ab}$ instead of $P_{a,b}$ when clear. 
\end{remark}

By Equation \eqref{eq:eval_prime_power}, note that $P_{ab}(1)\cdot Q_{ab}(1)=p^4\cdot q^2=(p^2q)^2$. Thus, the pair $P_{ab}$ and $Q_{2-a,2-b}$ would give generating functions for a $p^4$- and $q^2$-sided dice which are $(2,p^2q)$-solutions. For $a,b\in\{0,1,2\}$, we are curious to classify the ones that lead to valid solutions, that is, the cases where both polynomials have positive coefficients.

Since there are three values for both $a$ and $b$, there are nine possible pairs of functions to consider. Table \ref{tab:positive} gives a summary of the values for $a$ and $b$ where both $P_{ab}$ and $Q_{2-a,2-b}$ have positive coefficients, with one conditional case.

\begin{table}[h]

\begin{center}
\begin{tabular}{|ccccc|c|}\hline
$a$ & $b$ & $\leftrightarrow$ & $\alpha$ & $\beta$& Condition\\
\hline
$0$ & $0$ & $\leftrightarrow$ & $2$ & $2$&\\
$0$ & $1$ & $\leftrightarrow$ & $2$ & $1$&\\
$0$ & $2$ & $\leftrightarrow$ & $2$ & $0$&\\
$1$ & $1$ & $\leftrightarrow$ & $1$ & $1$&\\
$1$ & $2$ & $\leftrightarrow$ & $1$ & $0$&\\
$2$ & $2$ & $\leftrightarrow$ & $0$ & $0$&\\
$2$ & $1$ & $\leftrightarrow$ & $0$ & $1$ & $p=2,\ q\equiv 1\mod 4$\\
\hline
\end{tabular}

\caption{Cases where $P_{ab}$ and $Q_{\alpha\beta}$ has positive coefficients.} \label{tab:positive}
\end{center}
\end{table}

%\cooper{Adjust the following statement for cases that are always positive}

\begin{theorem}\label{thm:positive}
The first six rows of Table \ref{tab:positive} provides cases $P_{ab}\text{ and }Q_{\alpha\beta} \text{ always have positive coefficients. }$

%Further, $Q_{01}$ is conditionally positive when $p=2$ and $q\equiv 3 \mod 4$
\end{theorem}

\begin{proof} 
We provide the proof of three cases here. Remaining cases can be found in Appendix \ref{sec:ids}. Throughout we regularly utilize known cyclotomic identities provided in Section \ref{sec:cyc}. Our standard technique will be to rewrite each function as a product of evaluations of either $\phi_p(x)$ or $\phi_q(x)$, which will immediately have positive coefficients due to Equation \eqref{eq:cyc_prime}.

First, we demonstrate the pair $P_{00}$ and $Q_{22}$ both have positive coefficients. 

$$P_{00}(x)=x{\phi}_{p}^2{\phi}_{p^2}^2=x\phi_{p}^2(x)\phi_{p}^2(x^p),$$
which has positive coefficients.

$$Q_{22}(x)=x\phi_q^2\phi_{pq}^2\phi_{p^2q}^2=x(\phi_q\phi_{pq})^2\phi_{pq}^2(x^p).$$
Equation \eqref{eq:divisor} 
%and Proposition \ref{prop:cyc_mult} 
begets
$$Q_{22}(x)=x\phi_q^2(x^p)\frac{\phi_{q}^2(x^{pq})}{\phi_q^2(x^p)}=x\phi_q^2(x^{pq})$$
which again has positive coefficients.

We next demonstrate the pair $P_{02}(x)$ and $Q_{20}(x)$ both have positive coefficients. First, note
$$P_{02}(x)=x\phi_p^2\phi_{p^2}^2\phi_{p^2q}^2=x\phi_{p}^2(\phi_{p^2}\phi_{p^2q})^2.$$
Equation \eqref{eq:divisor} shows
$$P_{02}(x)=x\phi_p^2\phi_{p^2}^2(x^q)$$
which after using Equation \eqref{eq:cyc_prime_power_mult} becomes
$$P_{02}(x)=xp\phi_p^2\phi_p^2(x^{pq}).$$

Now note

$$Q_{20}(x)=x\phi_q^2\phi_{pq}^2=x(\phi_q\phi_{pq})^2.$$
After using Equation \eqref{eq:divisor}, we get
$$Q_{20}(x)=x\phi_q^2(x^p).$$

Finally, we demonstrate the pair $P_{11}(x)$ and $Q_{11}(x)$ have positive coefficients. Similar identities will be used as in the prior two cases, so we omit citing these for this case. First, note
$$P_{11}(x)=x\phi_p^2\phi_{p^2}^2\phi_{pq}\phi_{p^2q}=x(\phi_p(\phi_p\phi_{pq})(\phi_{p^2}(\phi_{p^2}\phi_{p^2q}))=x\phi_p\phi_p(x^q)\phi_{p^2}\phi_{p^2}(x^q).$$

Next, note
$$Q_{11}(x)=x\phi_q^2\phi_{pq}\phi_{p^2q}=x(\phi_{q}\phi_{pq})(\phi_{q}\phi_{p^2q})=x\phi_{p}(x^q)\phi_{q}(x^{p^2})$$
as desired.
\end{proof}

Thus, we now have six different ways of taking a pair of $p^2q$-sided dice, and finding alternative dice with $p^4$ and $q^2$ sides with the same sum frequency.

\begin{example}
We let $p=2$ and $q=3$. One can verify in this case that 
\[A_{01}=x\phi_2^2\phi_{4}^2\phi_{12}=x + 2 x^2 + 2 x^3 + 2 x^4 + x^5 + x^7 + 2 x^8 + 2 x^9 + 2 x^{10} + x^{11}\]
and 
\[B_{21}=x\phi_3^2\phi_{6}^2\phi_{12}=x + x^3 + 2 x^5 + x^7 + 2 x^9 + x^{11} + x^{13} \]
Thus, these are generating functions for a pair of dice which are $(2,12)$-solutions, as one can verify with the following table.

\begin{center}
\begin{tabular}{c|ccccccccc}
 & $1$ & $3$ & $5$ & $5$ & $7$ & $9$ & $9$ & $11$ & $13$ \\
\hline
$1$  & $2$  & $4$  & $6$  & $6$  & $8$  & $10$ & $10$ & $12$ & $14$ \\
$2$  & $3$  & $5$  & $7$  & $7$  & $9$  & $11$ & $11$ & $13$ & $15$ \\
$2$  & $3$  & $5$  & $7$  & $7$  & $9$  & $11$ & $11$ & $13$ & $15$ \\
$3$  & $4$  & $6$  & $8$  & $8$  & $10$ & $12$ & $12$ & $14$ & $16$ \\
$3$  & $4$  & $6$  & $8$  & $8$  & $10$ & $12$ & $12$ & $14$ & $16$ \\
$4$  & $5$  & $7$  & $9$  & $9$  & $11$ & $13$ & $13$ & $15$ & $17$ \\
$4$  & $5$  & $7$  & $9$  & $9$  & $11$ & $13$ & $13$ & $15$ & $17$ \\
$5$  & $6$  & $8$  & $10$ & $10$ & $12$ & $14$ & $14$ & $16$ & $18$ \\
$7$  & $8$  & $10$ & $12$ & $12$ & $14$ & $16$ & $16$ & $18$ & $20$ \\
$8$  & $9$  & $11$ & $13$ & $13$ & $15$ & $17$ & $17$ & $19$ & $21$ \\
$8$  & $9$  & $11$ & $13$ & $13$ & $15$ & $17$ & $17$ & $19$ & $21$ \\
$9$  & $10$ & $12$ & $14$ & $14$ & $16$ & $18$ & $18$ & $20$ & $22$ \\
$9$  & $10$ & $12$ & $14$ & $14$ & $16$ & $18$ & $18$ & $20$ & $22$ \\
$10$ & $11$ & $13$ & $15$ & $15$ & $17$ & $19$ & $19$ & $21$ & $23$ \\
$10$ & $11$ & $13$ & $15$ & $15$ & $17$ & $19$ & $19$ & $21$ & $23$ \\
$11$ & $12$ & $14$ & $16$ & $16$ & $18$ & $20$ & $20$ & $22$ & $24$
\end{tabular}
\end{center}
\end{example}

%For all remaining cases, see the appendix for the resulting proofs.
%\\

We save the last entry of Table \ref{tab:positive} to be the last case we deal with; this discussion starts with Theorem \ref{thm:q01_coeffs}. For now, we turn our attention to other two remaining pairs of $P_{ab}$ and $Q_{2-a,2-b}$ to discuss. In two cases, $Q_{2-a,2-b}$ will always have two negative coefficients, as our next two Lemmas demonstrate.

\begin{lemma}
$Q_{12}$ has a negative coefficient. 
\end{lemma}
\begin{proof} % should the summations be squared?
After simplifying and writing the rational expressions in series notation, we have
\[Q_{12}=x(1-x^q)(1-x^p)(1-x^{p^2q})^2 \left(\sum_{i\geq 0}x^{i}\right)\left(\sum_{i\geq 0}(i+1)x^{p^2i}\right)\left(\sum_{i\geq 0}x^{pqi}\right)\]

We proceed by cases depending on the value of $\max(p,q)$. 
If $\max(p,q)=p$, observe 
\[[x^p](Q_{12}/x)=-1-1+1=-1.\]
These coefficient come only from the factors $(1-x^q)$, $(1-x^p)$, and $\displaystyle \sum_{i\geq 0}x^{i}$.

Suppose otherwise that $\max(p,q)=q$. We claim in this case that there exists an integer $k$ so that $0\leq k<p$ so that $[x^{q+k}](Q_{12}/x)=-1$. Note here it is possible to have nonzero multiples of $p^2$ smaller than $q$, so in addition to the aforementioned factors, we additionally have $\displaystyle \left(\sum_{i\geq 0}(i+1)x^{p^2i}\right)$ as a possible factors, but no others as for our chosen $k$, $q+k<pq<p^2q$. Terms in this factor may contribute to the coefficient $x^{q+k}$ when multiplied by the appropriate term of $\displaystyle \sum_{i\geq 0}x^{i}$ alone or the appropriate term of $\displaystyle \sum_{i\geq 0}x^{i}$ along with the $-x^p$ factor appearing in $(1-x^p)$. That is, 

\begin{align*}
[x^{q+k}](Q_{12}/x)&=[x^{q+k}](1-x^q)(1-x^p)\left(\sum_{i\geq 0}x^{i}\right)\left(\sum_{i\geq 0}(i+1)x^{p^2i}\right)\\
&=[x^{q+k}]\left(-x^qx^k+\sum_{i=0}^{\left\lfloor {q+k\over p^2}\right\rfloor}(i+1)x^{ip^2}x^{q+k-ip^2}- \sum_{i=0}^{\left\lfloor {q+k-p\over p^2}\right\rfloor}(i+1)x^{ip^2}x^{q+k-p-ip^2}     \right)\\
&=\left(-1+\sum_{i=0}^{\left\lfloor {q+k\over p^2}\right\rfloor}(i+1)- \sum_{i=0}^{\left\lfloor {q+k-p\over p^2}\right\rfloor}(i+1)    \right).
\end{align*}

By Lemma \ref{lemma:qpk}, there exists a $k$ so that $0\leq k<p$ and $\lfloor {q+k\over p^2}\rfloor=\lfloor {q+k-p\over p^2}\rfloor$. Thus, for this $k$, the above equals $-1$.

\end{proof}

%\[\sum_{i\geq 0} x^{ip^2}\]

% 7/18 work on q+k together
% Over Weekend:
% Cooper: p>q proof
% Evelyn: q>p proof

\begin{lemma} 
$Q_{02}$ has a negative coefficient.
\end{lemma}
\begin{proof}
After rewriting, we have
$$Q_{02}=x(1-x^q)^2(1-x^p)^2(1-x^{p^2q})^2\left(\sum_{i\geq 0}(i+1)x^i\right)\left(\sum_{i\geq 0}(i+1)x^{p^2i}\right)\left(\sum_{i\geq 0}(i+1)x^{pqi}\right)$$
We look at the cases depending on $\max(p,q)$. If $\max(p,q)=p$, the coefficients come only from the factors $(1-x^q)^2$, $(1-x^p)^2$, and $\displaystyle \sum_{i\geq 0}(i+1)x^{i}$. We claim that in this case, the $(p+q+1)$th coefficient will always be negative,
Using these three factors, there are three ways to build the power $p+q-1$. The first uses $(1-x^q)^2$ and $\displaystyle \sum_{i\geq 0}(i+1)x^{i}$, the second uses $(1-x^p)^2$ and $\displaystyle \sum_{i\geq 0}(i+1)x^{i}$ and the third uses $\displaystyle \sum_{i\geq 0}(i+1)x^{i}$. Taking the coefficients from these, we get $-2p$, $-2q$, and $p+q$, respectively, whose sum is $-p-q$, as desired. 

Otherwise, if $\max(p,q)=q$, we assert that there is an integer $k$, where $0\leq k \leq p$, so that $[x^{q+k}](Q_{02}/x)<0$. We must note that it is possible to have non-zero multiples of $p^2$ smaller than $q$, meaning $\displaystyle\sum_{i\geq 0}(i+1)x^{p^2i}$ is a possible factor. That is,

\begin{align*}
[x^{q+k}](Q_{02}/x) &= [x^{q+k}](1-x^q)^2(1-x^p)^2\left(\sum_{i\geq 0}(i+1)x^i\right) \left(\sum_{i\geq 0}(i+1)x^{p^2i}\right)\\
&=[x^{q+k}]\Bigg(-2x^q(k+1)x^k+\left\lfloor \frac{k}{p}\right\rfloor 4x^px^q+\sum_{i=0}^{\left\lfloor \frac{q+k}{p^2}\right\rfloor}(i+1)x^{ip^2}(q+k-ip^2+1)x^{q+k-ip^2} \\
&+ \sum_{i=0}^{\left\lfloor \frac{q+k-2p}{p^2}\right\rfloor}(i+1)x^{ip^2}x^{2p}(q+k-ip^2-2p+1)x^{q+k-ip^2-2p} \\
&- 2\sum_{i=0}^{\left\lfloor\frac{q+k-p}{p^2}\right\rfloor}(i+1)x^{ip^2}x^p(q+k-ip^2-p+1)x^{q+k-ip^2-p} \Bigg)\\
&=-2(k+1)+4\left\lfloor\frac{k}{p}\right\rfloor+\sum_{i=0}^{\left\lfloor\frac{q+k}{p^2}\right\rfloor}(i+1)(q+k-ip^2+1)+\sum_{i=0}^{\left\lfloor\frac{q+k-2p}{p^2}\right\rfloor}(i+1)(q+k-ip^2-2p+1)\\
&-2\sum_{i=0}^{\left\lfloor\frac{q+k-p}{p^2}\right\rfloor}(i+1)(q+k-ip^2-p+1).
\end{align*}
If $p>2$, using Corollary \ref{cor:qpk}, there exists a $k$ for which $\left\lfloor {q+k-2p\over p^2}\right\rfloor=\left\lfloor {q+k-p\over p^2}\right\rfloor=\left\lfloor {q+k\over p^2}\right\rfloor$, and thus  the above simplifies to
\[ -2(k+1)+4\left\lfloor\frac{k}{p}\right\rfloor\]
which is $-2(k+1)$ for $k<p$ and $-2p+2$ for $k=p$, both of which being negative.

 Suppose instead $p=2$. For this value of $p$, $2p=p^2$, and so $\left\lfloor {q+k\over p^2}\right\rfloor=\left\lfloor {q+k-2p\over p^2}\right\rfloor+1.$ This time, select a $k$ using Lemma \ref{lemma:qpk} so that 
$\left\lfloor{q+k-p\over 4}\right\rfloor=\left\lfloor{q+k\over 4}\right\rfloor$. 
In this case, simplifies to 
\[-2(k+1)+4\left\lfloor{k\over 2}\right\rfloor+p-\left\lfloor{q+k-p\over p^2}\right\rfloor-1=-2(k+1)+4\left\lfloor{k\over 2}\right\rfloor-\left\lfloor{q+k-6\over 4}\right\rfloor\]
which will always be at least $-1$ since $q\geq 3$.

% we can show that all components of the remaining summations will be negative. Since $k$ can at most ever be $p$, then it shows that $-2(k+1)+4\lfloor\frac{k}{p}\rfloor$ will always be negative as well. 
Thus, we may conclude that, regardless of the values for $p$ and $q$, $Q_{02}$ will always have a negative coefficient.

\end{proof}
%%I think it's primes that are 1 more than a multiple of 4? Alternatively, we say p mod 4=1.
%\begin{corollary}
%$Q_{01}$ will always be negative, except when $q\equiv 1\pmod{4}$
%\end{corollary}

The final case is for $P_{21}$ and $Q_{01}$. In this case, $P_{21}$ always has positive coefficients by Lemma \ref{lemma:P_21}, so it suffices to focus on $Q_{01}$. Interestingly, in this case, we have conditional positivity, as described by the following. 

\begin{con}\label{con:q01}
$Q_{01}$ has nonnegative coefficients if and only if $p=2$ and $q\equiv 1 \mod 4$.
\end{con}

We have ample numerical evidence to justify the case where $p\neq 2$, but were not able to fully classify this case. However, we can prove this conjecture in the case of $p=2$.

\begin{theorem}\label{thm:q01_coeffs} Let $p=2$. Then $Q_{01}$ has nonnegative coefficients if and only if $q\equiv 1 \mod 4$.
\end{theorem}

The argument for positivity even in this specific case requires substantive exploration. We start by providing a formulation of $[x^k]Q_{01}/x$ in this case. 

\begin{proposition}
If $p=2$, then

$$[x^k]\frac{Q_{01}}{x}=\sum_{i=\max(0,\lceil (k-q+1)/2\rceil)}^{\lfloor {k/2}\rfloor}(k-2i+1)(-1)^i + \sum_{i=\max(0,\lceil (k+2-2q )/2\rceil)}^{\lfloor(k-q)/2\rfloor}(2q+2i-k-1)(-1)^i$$\label{prop:q01_form}
\end{proposition}
\begin{proof}
By definition, 
$$\frac{Q_{01}}{x} =\phi_q^2\phi_{p^2q}.$$

Letting $p=2$, we have $\frac{Q_{01}}{x} =\phi_q^2\phi_{4q}$, which is the same as $\phi_q^2\phi_{2q}(x^2)$ by Equation \eqref{eq:cyc_prime_power_mult}. By Equation \eqref{eq:parity}, this becomes
$\phi_q^2\phi_q(-x^2)$.
Thus, representing $\phi_q$ as series, we have
\begin{align*}
\frac{Q_{01}}{x}&=\left( \sum_{j=0}^{q-1}x^j\right)^2\left(\sum_{i=0}^{q-1}(-1)^ix^{2i}\right)\\
&=\left(\sum_{j=0}^{q-1}(j+1)x^j + \sum_{j=1}^{q-1}(q-j)x^{q+j-1}\right) \left(\sum_{i=0}^{q-1}(-1)^ix^{2i}\right)\\
&=\sum_{i=0,j=0}^{q-1}(j+1)(-1)^ix^{j+2i} + \sum_{j=1,i=0}^{q-1}(q-j)(-1)^ix^{q+j+2i-1}.
\end{align*}

Thus, we have that 
\begin{align*}
[x^k]\frac{Q_{01}}{x}&=\sum_{i=\max(0,\lceil (k-q+1)/2\rceil)}^{\lfloor {k/2}\rfloor}(k-2i+1)(-1)^i + \sum_{i=\max(0,\lceil (k+2-2q )/2\rceil)}^{\lfloor(k-q)/2\rfloor}(2q+2i-k-1)(-1)^i,
\end{align*}
where the bounds on the first summation are derived from the fact that $0\leq j\leq q-1$ and the substitution $j=k-2i$, giving us that \[{k-q+1\over 2}\leq i\leq {q\over 2},\]
and the bounds on the second one come similarly, except in this case $1\leq j\leq q-1$, and the substitution $j=k-q-2i+1$, giving

\[{k+2-2q \over 2} \leq i\leq {k-q\over 2}.\]
\end{proof}

Using this formula, we can demonstrate the circumstances where, provided $p=2$, we have nonnegative coefficients. Note that cyclotomic polynomials are palindromic and products of palindromic polynomials are again palindromic; thus, given the degree of this polynomial is $4q-4$, we need only demonstrate what we know about the coefficients of the terms up to degree $2q-2$. 

We proceed by grouping the coefficients into three categories: $0\leq k<q$, $q\leq k<2q-2$, and $k=2q-2$. 

\begin{lemma}
For $0\leq k<q$, the coefficients of $Q_{01}$ are all positive. 
\end{lemma}
\begin{proof}
By Proposition \ref{prop:q01_form}, in this case we have 
$$[x^k]\frac{Q_{01}}{x}=\sum_{i=0}^{\lfloor {k/2}\rfloor}(k-2i+1)(-1)^i.$$

Observe this is an alternating sum, where the largest term (when $i=0$) is positive, and so the resulting sum must be positive.

\end{proof}

Before moving on to our next case, we bring to the reader's attention some notation we will use in the remainder of this section, starting with the next result. We let $\delta(\bullet)$ be the \textbf{Kronecker delta} function, which returns $1$ when the input is true and $0$ otherwise.
\begin{lemma}
For $q\leq k<2q-2$, the coefficients of $Q_{01}$ are all nonnegative. 
\end{lemma}
\begin{proof}
In this case, using Proposition \ref{prop:q01_form}, we get 
\begin{align*}
[x^k]\frac{Q_{01}}{x}&=\sum_{i=\lceil (k-q+1)/2\rceil}^{\lfloor {k/2}\rfloor}(k-2i+1)(-1)^i + \sum_{i=0}^{\lfloor(k-q)/2\rfloor}(2q+2i-k-1)(-1)^i\\
&=\sum_{i=0}^{\lfloor {k/2}\rfloor}(k-2i+1)(-1)^i -\sum_{i=0}^{
\lceil (k-q-1)/2\rceil}(k-2i+1)(-1)^i + \sum_{i=0}^{\lfloor(k-q)/2\rfloor}(2q+2i-k-1)(-1)^i\\
&=\sum_{i=0}^{\lfloor {k/2}\rfloor}(k-2i+1)(-1)^i + 2\sum_{i=0}^{\lfloor(k-q)/2\rfloor}(q-k+2i-1)(-1)^i,\\
&=\sum_{i=0}^{\lfloor {k/2}\rfloor}k(-1)^i-\sum_{i=0}^{\lfloor {k/2}\rfloor}(2i-1)(-1)^i + 2\sum_{i=0}^{\lfloor(k-q)/2\rfloor}(q-k)(-1)^i+ 2\sum_{i=0}^{\lfloor(k-q)/2\rfloor}(2i-1)(-1)^i,\\
&=\sum_{i=0}^{\lfloor {k/2}\rfloor}k(-1)^i-\sum_{i=0}^{\lfloor {k/2}\rfloor-1}(2i+1)(-1)^i + 2\sum_{i=0}^{\lfloor(k-q)/2\rfloor}(q-k)(-1)^i+ 2\sum_{i=0}^{\lfloor(k-q)/2\rfloor-1}(2i+1)(-1)^i-1,
\end{align*}
where the third-to-last equality follows from Proposition \ref{prop:floor2ceil}.

This formulation can be simplified significantly. To do so, we let $\delta(P)$ be the indicator function of a statement $P$ which equals $1$ when $P$ is true and $0$ otherwise. Further, we let $\alpha=\lfloor {k\over 2}\rfloor$ and $\beta=\lfloor{k-q\over 2}\rfloor$. Then the above, after applying Proposition \ref{prop:alt_odd} and simplification, yields 
$$
\delta(\alpha \text{ even}) k -(-1)^{\alpha}\alpha +\delta(\beta \text{ even})(2q-2k)+(-1)^\beta 2\beta -1
.$$

We claim this expression is always nonnegative, which we have elected to offload to Lemma \ref{lem:q01_p2_comp} in the Appendix, as the details are case-based computations.

\end{proof}

%\george{I recommend we make the following about the $2k-2$ coefficient. Something like ``this coefficient is nonnegative if and only if $q$ is congruent to 3" or something.}

\begin{lemma}
For $p=2$, we have
$$[x^{2q-2}] Q_{01} = (-1)^{{q-1}\over2}. $$
In particular, $[x^{2q-2}] Q_{01}$ has a negative coefficient if and only if $q\equiv 3\mod 4$.
\end{lemma}

\begin{proof}

By Proposition \ref{prop:q01_form}

\begin{align*}
[x^{2q-2}]\frac{Q_{01}}{x}&=\sum_{i=\lceil (q-1)/2\rceil}^{q-1}(2q-2i-1)(-1)^i + \sum_{i=0}^{\lfloor(q-2)/2\rfloor}(2i+1)(-1)^i\\
&=\sum_{i=0}^{q-1}(2q-2i-1)(-1)^i -\sum_{i=0}^{\lfloor (q-2)/2\rfloor}(2q-2i-1)(-1)^i + \sum_{i=0}^{\lfloor(q-2)/2\rfloor}(2i+1)(-1)^i\\
&=\sum_{i=0}^{q-1}(2q-2i-1)(-1)^i  + \sum_{i=0}^{\lfloor(q-2)/2\rfloor}(4i+2-2q)(-1)^i\\
&=\sum_{i=0}^{q-1}2q(-1)^i-\sum_{i=0}^{q-1}(2i+1)(-1)^i  + 2\sum_{i=0}^{\lfloor(q-2)/2\rfloor}(2i+1)(-1)^i-\sum_{i=0}^{\lfloor(q-2)/2\rfloor}2q(-1)^i\\
&=2q-(-1)^{q-1}q  + 2(-1)^{\lfloor(q-2)/2\rfloor}\lfloor q/2\rfloor-\sum_{i=0}^{\lfloor(q-2)/2\rfloor}2q(-1)^i\\
&=q + (-1)^{\lfloor(q-2)/2\rfloor}(q-1)-\sum_{i=0}^{\lfloor(q-2)/2\rfloor}2q(-1)^i\\
\end{align*}

Since $q\not=p=2$, $q$ is odd. Therefore, $q=2k+1$ for $k\in \mathbb{N}$. Observe $\lfloor\frac{q-2}{2}\rfloor$. Then $\lfloor\frac{q-2}{2}\rfloor= \lfloor\frac{2k-1}{2}\rfloor=\lfloor k-\frac{1}{2}\rfloor$. Looking back at $q=2k+1$, then $k=\frac{q-1}{2}$. Since $q$ has odd parity, $k$ must be an integer. So, $\lfloor k-\frac{1}{2}\rfloor= k+\lfloor-\frac{1}{2}\rfloor= k-1$. Thus, $\lfloor\frac{q-2}{2}\rfloor=k-1$.
\\

\begin{align*}
[x^{2q-2}]\frac{Q_{01}}{x}&=q + (-1)^{\lfloor(q-2)/2\rfloor}(q-1)-\sum_{i=0}^{\lfloor(q-2)/2\rfloor}2q(-1)^i\\
&=q+(-1)^{k-1}(q-1)-\sum_{i=0}^{k-1}2q(-1)^i\\
&=q-(-1)^k(q-1)-\sum_{i=0}^{k}2q(-1)^i+2q(-1)^{k}\\
&=q-(-1)^k(q-1)+(-1)^k(2q)-\sum_{i=0}^{k}2q(-1)^i\\
&=q+(-1)^k(q+1)-\sum_{i=0}^{k}2q(-1)^i
\end{align*}

Case 1: $k$ is even
\begin{align*}
[x^{2q-2}]\frac{Q_{01}}{x}&=q+(-1)^k(q+1)-\sum_{i=0}^{k}2q(-1)^i\\
&=2q+1-\sum_{i=0}^{k}2q(-1)^i
\end{align*}
since $k$ is even, the sum has odd terms, thus the sum becomes $-2q$. So,
$$[x^{2q-2}]\frac{Q_{01}}{x}=-2q+1-(-2q)=1$$

Case 2: $k$ is odd
\begin{align*}
[x^{2q-2}]\frac{Q_{01}}{x}&=q+(-1)^k(q+1)-\sum_{i=0}^{k}2q(-1)^i\\
&=-1-\sum_{i=0}^{k}2q(-1)^i
\end{align*}
since $k$ is even, the sum has even terms, thus the sum becomes $0$. So,
$$[x^{2q-2}]\frac{Q_{01}}{x}=-1$$

Therefore, $[x^{2q-2}]\frac{Q_{01}}{x}=(-1)^k=(-1)^{q-1\over 2}$

%$q\equiv 3\pmod{4}$ iff $[x^{2q-2}](\frac{Q_{01}}{x})=-1$ %can also be less than zero, then show it's -1

%We propose that there's an equation that models the sum as follows: 
%$$\sum_{k=0}^n(2k+1)(-1)^k = (-1)^n(n+1)$$
%We will let $\displaystyle S_n=\sum_{k=0}^n(2k-1)(-1)^k$.
%We will continue with induction. \\This holds for the base case when $n=0$
%$$(2(0)+1)(-1)^0=1 \implies (-1)^0(0+1)=1$$
%Assume $\exists m\in\mathbb{N}$ such that $(S_m)=(-1)^m(m+1)$. Observe the $(S_{m+1})$ case 

%\evelyn{TODO: finish Lemma}
%x^{2q-2}of Q_{01} = (-1)^{{q-1}\over2} show the modulo equations get an even or odd power
\end{proof}

At the start of this section, we mentioned that there were multiple ways of constructing $(2,p^2q)$-solutions. Throughout this section, we only focused on solutions where one dice had $p^4$ sides and the other had $q^2$ sides. We pose now the following. 

\begin{problem}\label{prob:open}
Can one find a pair of $p^2q^2$- and $p^2$-sided dice which are $(2,p^2q)$-solution?
\end{problem}

What is interesting about this case is that there are two pairs of polynomials to explore: either
\[x{{\phi}_{p}^2{\phi}_{pq}^a{\phi}_{p^2q}^b \text{ and }x\phi}_{p^2}^2{\phi}_{q}^2{\phi}_{pq}^a{\phi}_{p^2q}^b,\]
or
\[x{\phi}_{p^2}^2{\phi}_{pq}^a{\phi}_{p^2q}^b \text{ and }x{\phi}_{p}^2{\phi}_{q}^2{\phi}_{pq}^a{\phi}_{p^2q}^b.\]

Together, the pairs offers $18$ more cases that one could explore. 
%\pagebreak

%\pagebreak

\section{Square-sided dice for a pair of $pqr$-sided dice}\label{sec:pqr}
In this section, we provide preliminary exploration of $(2,pqr)$-solutions. In this case, the corresponding generating function for the frequency sum is
$$x^2{(1-x^{pqr})^2\over (1-x)^2}=x^2\phi_p^2\phi_q^2\phi_r^2\phi_{pq}^2\phi_{pr}^2\phi_{qr}^2\phi_{pqr}^2.$$
We propose that we can create two dice, a $(pq)^2$-sided and a $r^2$-sided dice, which are $(2,pqr)$-solutions. 

We start by defining our generating functions in this case.

\begin{definition}
%% Should we change a and b in Q to alpha and beta?
For $a,b,c,d\in \{0,1,2\}$, we let
$$A_{a,b,c,d}(x)=x\phi_p^2\phi_q^2\phi_{pq}^a\phi_{pr}^b\phi_{qr}^c\phi_{pqr}^d$$
and
$$B_{a,b,c,d}(x)=x\phi_r^2\phi_{pq}^a\phi_{pr}^b\phi_{qr}^c \phi_{pqr}^d.$$

\end{definition}
\noindent As before, we more often omit the commas from the notation.
Note that \[A_{abcd}(1)\cdot B_{abcd}(1)=p^2q^2\cdot r^2=(pqr)^2,\] thus, the pair $A_{abcd}$ and $B_{2-a,2-b,2-c,2-d}$ would be for generating functions $(2,pqr)$-solutions, provided there coefficients are positive. 

%We introduce variables $\alpha, \beta,\gamma,\delta$ to be used for $B$ in place of $a,b,c,d$,  will be remaining out of the $2$ values; i.e $\alpha=2-a$, $\beta=2-b$, and so on so forth.
%$$A_{abcd}(x)=x\phi_p^2\phi_q^2\phi_{pq}^a\phi_{pr}^b\phi_{qr}^c\phi_{pqr}^d$$
%$$B_{\alpha\beta\gamma\delta}(x)=x\phi_r^2\phi_{pq}^\alpha\phi_{pr}^\beta\phi_{qr}^\gamma \phi_{pqr}^\delta$$

%Using Proposition \ref{prop:expanded_cyclotomic} we can rewrite and simplify these functions to:
%$$A(x)=\left(\frac{1-x^p}{1-x}\right)^2\left(\frac{1-x^q}{1-x}\right)^2 \left(\frac{(1-x^{pq})(1-x)}{(1-x^p)(1-x^q)}\right)^a \left(\frac{(1-x^{pr})(1-x)}{(1-x^p)(1-x^r)}\right)^b\left(\frac{(1-x^{qr})(1-x)}{(1-x^q)(1-x^r)}\right)^c$$
%$$A(x)=\frac{(1-x^{pqr})(1-x^p)(1-x^q)(1-x^r)}{(1-x^{pq})(1-x^{pr})(1-x^{qr})(1-x)}$$
%$$A(x)=(1-x^p)^{2-a-b+d}(1-x^q)^{2-a-c+d}(1-x^{pq})^{a-d}(1-x^{pr})^{b-d} (1-x^{qr})^{c-d}(1-x)^{-4+a+b+c-d}$$

Since there are three values for each of $a$, $b$, $c$, and $d$, there are $81$ possible pairs to study. Thus, we offer in this section only an early, preliminary study, which we hope the reader is curious enough to expand on. 
One fruitful thing here is that we can shorten this by coming up with necessary conditions for having negative coefficients. This can be done by using \cite[Lemma 4.1]{firstpaper}, which we recall here.
\begin{lemma}[\cite{firstpaper}, Lemma 4.1]\label{lem:shriram}
Let $F(x)$ be a function of the form 
\[F(x)=\prod_{i=1}^k (1-x^{n_i})^{\epsilon_i}\]
where the $n_i$ are distinct positive integers and $\epsilon_i\in \mathbb{Z}$. If there exists a $j$ for which $n_j=1$ and $\epsilon_j>0$, then $[x]F<0$.
\end{lemma}

Using \cite[Corollary 5.1]{firstpaper} as a guide, and the cyclotomic polynomials for $A_{abcd}$ and $B_{abcd}$ in rational form as in Equations \eqref{eq:cyc_prime} and \eqref{eq:pqr}, one can deduce the following.
\begin{cor}\label{cor:old}
$$A_{abcd}\text{ has a negative coefficient when } a+b+c>4+d$$
and 
$$B_{abcd}\text{ has a negative coefficient when } a+b+c>2+d.$$
\end{cor}
\begin{proof}
We prove the inequality for $A_{abcd}$; the proof for $B_{abcd}$ is similar. Note that for $A_{abcd}$, when writing its cyclotomic polynomials in rational form, the only ones with a multiple of $(1-x)$ on the denominator are $\phi_p$, $\phi_q$, and $\phi_{pqr}$; thus, a total of $4+d$ multiples of $(1-x)$ are in the denominator for $A_{abcd}$. The remaining cycltomic polynomials, of which there are $a+b+c$, each have a multiple of $(1-x)$ on the numerator. Thus in total there are $a+b+c-(4+d)$ multiples of $(1-x)$ on the numerator of $A_{abcd}$ when simplified.
\end{proof}

One can use various methods (brute force, stars-and-bars/balls-and-bins inspired counting strategies, or computational tools) to verify the following.
\begin{proposition}
There are $36$ cases in which Corollary \ref{cor:old} produces a pair of polynomials, $A_{abcd}$ and $B_{2-a,2-b,2-c,2-d}$, with at least one having a negative coefficient.
\end{proposition}

%While one can use standard balls-and-bins (also known as stars-and-bars) counting techniques to count many cases, there will be exclusionary cases that make this tedious to do by hand. 

 %do we need to show this?
 
 As for the remaining $45$ cases, we believe Table \ref{tab:pos2} represents the only possible cases where both polynomials yield positive coefficients, with the last row having a condition as we saw with $Q_{01}$ in the last section. 
 
\begin{table}[h]
\begin{center}
\begin{tabular}{|ccccccccc|c|}\hline
$a$ & $b$ & $c$ & $d$ & $\leftrightarrow$ & $\alpha$ & $\beta$ & $\gamma$ & $\delta$ & condition \\
\hline
$2$ & $0$ & $0$ & $0$ & $\leftrightarrow$ & $0$ & $2$ & $2$ & $2$ & \\
$2$ & $1$ & $0$ & $1$ & $\leftrightarrow$ & $0$ & $1$ & $2$ & $1$ & \\
$2$ & $0$ & $1$ & $1$ & $\leftrightarrow$ & $0$ & $2$ & $1$ & $1$ & \\
$2$ & $1$ & $1$ & $1$ & $\leftrightarrow$ & $0$ & $1$ & $1$ & $1$ & \\
$2$ & $1$ & $1$ & $2$ & $\leftrightarrow$ & $0$ & $1$ & $1$ & $0$ & \\
$2$ & $2$ & $0$ & $2$ & $\leftrightarrow$ & $0$ & $0$ & $2$ & $0$ & \\
$2$ & $2$ & $1$ & $2$ & $\leftrightarrow$ & $0$ & $0$ & $1$ & $0$ & \\
$2$ & $1$ & $1$ & $2$ & $\leftrightarrow$ & $0$ & $1$ & $1$ & $0$ & \\
$2$ & $0$ & $2$ & $2$ & $\leftrightarrow$ & $0$ & $2$ & $0$ & $0$ & \\
$2$ & $1$ & $2$ & $2$ & $\leftrightarrow$ & $0$ & $1$ & $0$ & $0$ & \\
$2$ & $2$ & $2$ & $2$ & $\leftrightarrow$ & $0$ & $0$ & $0$ & $0$ & \\
$1$ & $1$ & $1$ & $1$ & $\leftrightarrow$ & $1$ & $1$ & $1$ & $1$ & $\max(p,q,r)=r, \min(p,q,r)=2$ \\

\hline
\end{tabular}
\end{center}
\caption{The cases where both $A_{abcd}$ and its pair $B_{\alpha\beta\gamma\delta}$ have positive coefficients.}\label{tab:pos2}
\end{table}

\begin{theorem}\label{thm:pqr}
The entries of Table \ref{tab:pos2} correspond to pairs $A_{abcd}$ and $B_{2-a,2-b,2-c,2-d}$ where both have positive coefficients.
\end{theorem}

There are two parts to this Theorem. For the first 11 rows of Table \ref{tab:pos2}, the positivity of the corresponding coefficients can be proven using various cyclotomic identities as done throughout this paper. Thus, we leave this to the reader to verify using said techniques.
We focus our energies instead on the last entry.

\begin{lemma}\label{lem:pqr}
$A_{1111}$ has positive coefficients for all $p,q,r$, while  $B_{1111}$ has positive coefficients if and only if $\max(p,q,r)=r$ and $\min(p,q,r)=2$.
\end{lemma}
\begin{proof}
First, we show $A_{1111}$ has positive coefficients. Recall
\[A_{1111}=\phi_p^2\phi_q^2\phi_{pq}\phi_{qr}\phi_{pr}\phi_{pqr}\]
Using our cyclotomic identities from Section \ref{sec:cyc} by grouping together $\phi_p\phi_{pr}$ and $\phi_q\phi_{qr}\phi_{pr} \phi_{pqr}$, we get 
\[A_{1111}=\phi_p\phi_q\phi_p(x^r)\phi_q(x^{pr})\]

We move on now to $B_{1111}$. We start with the case where $p=2$. We have
\begin{align*}
B_{1111}&=\phi_r^2\phi_{2q}\phi_{2r}\phi_{qr}\phi_{2qr}\\
&=\phi_r^2\phi_q(-x)\phi_r(-x)\phi_{qr}\phi_{qr}(-x)\\
&=\phi_r\phi_r(x^q)\phi_q(-x)\phi_{r}(-x^q).
\end{align*}

We claim that $\phi_r(x^q)\phi_r(-x^q)$ and $\phi_r\phi_q(-x)$ both have positive coefficients, the latter specifically being true in the case where $q<r$. 
This follows from verifying that
% \left(\sum_{k=0}^m x^k\right)\left(\sum_{k=0}^n (-x)^k\right)=
\[ \left(\sum_{k=0}^{m-1} x^k\right)\left(\sum_{k=0}^{n-1} (-x)^k\right)={1-x^m\over 1-x}{1+x^n\over 1+x}\]
has positive coefficients if and only if $n\leq m$, where $m$ and $n$ are odd integers.
If $n>m$, then
\[ [x^{m}] \left(\sum_{k=0}^{m-1} x^k\right)\left(\sum_{k=0}^{n-1} (-x)^k\right)=\sum_{k=1}^{m} (-1)^{k}=-1.\]

Otherwise, $n\leq m$, for all $0\leq j\leq m+n-2$, we have
\[ [x^j] \left(\sum_{k=0}^{m-1} x^k\right)\left(\sum_{k=0}^{n-1} (-x)^k\right)=\sum_{k=0}^{\min(j,n)} (-1)^k\geq 0.\]
By symmetry, the case where $q=2$ is similar.

When $p,q>2$, we can guarantee $B_{1111}$ has negative coefficients. In this case, 
\[B_{1111}=\phi_r\phi_{pq}\phi_r(x^{pq})=(1-x^r)(1-x^{pqr})\left(\sum_{k\geq 0} x^{pk} \right)\left(\sum_{k\geq 0} x^{qk} \right)\left(\sum_{k\geq 0} x^{pqk} \right).\]

Note that if $p,q>2$, the coefficient of $x^r$ will be $-1$, since $pk$, $qk$, and $pqk$ can never equal $r$, and further, summing such numbers together will always be even. This is to say that the only factor contributing to the coefficient of $x^r$ is $(1-x^r)$.
\end{proof}

%\george{TODO use computation to verify remaining cases all have negatives OR actually prove it. (Add conjecture) }

\begin{example}
Let $p=2$, $q=3$, and $r=5$. Using the last entry of Table \ref{tab:pos2}, we build two generating functions which give dice that are $(2,30)$-solutions: 
\begin{align*}
A_{1111}&=x\phi_2^2\phi_3^2\phi_6\phi_{10}\phi_{15}\phi_{30}\\
&=x + 2 x^2 + 2 x^3 + x^4 + x^6 + 2 x^7 + 2 x^8 + x^9 + x^{11} + 2 x^{12} + 2 x^{13} + x^{14} \\ &+ x^{16} + 2 x^{17} + 2 x^{18} + x^{19}+ x^{21} + 2 x^{22} + 2 x^{23} + x^{24} + x^{26} + 2 x^{27} + 2 x^{28} + x^{29} 
\end{align*}

and 
\begin{align*}
B_{1111}&=x\phi_5^2\phi_{10}\phi_{15}\phi_{30}\\
&=x + x^3 + x^4 + x^5 + 2 x^7 + x^9 + x^{10} + x^{11} + 2 x^{13} + x^{15}+ x^{16} \\ 
&+ x^{17} + 2 x^{19} + x^{21} + x^{22} + x^{23} + 2 x^{25} + x^{27} + x^{28} + x^{29} + x^{31} 
\end{align*}
We omit the table for this example, as it would have 900 entries.
\end{example}

We have done extensive testing providing evidence of the following, which we hope the reader is curious enough to explore and provide proofs for.

\begin{con}\label{con:pqr}
In the remaining 33 pairs of the polynomials $A_{abcd}$ and $B_{2-a,2-b,2-c,2-d}$ (including the case where $p\neq 2$ for when $a=b=c=d=1$), one of the two polynomials will have a negative coefficient.
\end{con}

%We further propose the following problem which we hope the reader may be eager to explore. 

%\newpage

\section*{Appendix}

\appendix

\section{Additional Results for the $p^2q$ Case}\label{sec:ids}

Below include remaining cases left unproven in Section \ref{sec:psquaredq} in Theorem \ref{thm:positive}. As we were thorough in citing which results from Section \ref{sec:cyc} were useful, we omit using these here.

\begin{lemma}\label{lemma:P_01}
$P_{01}(x)$ and $Q_{21}(x)$ have positive coefficients.
\end{lemma}

\begin{proof}
First, note 
$$P_{01}(x)=x\phi_p^2\phi_{p^2}^2\phi_{p^2q}=x\phi_p^2\phi_{p^2}^2\phi_{pq}(x^p)=x\phi_p\phi_{p^2}^2\phi_p(x^q).$$

Next, note
$$Q_{21}(x)=x\phi_q^2\phi_{pq}^2\phi_{p^2q}=x\phi_q^2(x^p)\phi_{pq}(x^p)=x\phi_q(x^p)\phi_q(x^{p^2}).$$
\end{proof}

\begin{lemma}\label{lemma:P_12}
$P_{12}(x)$ and $Q_{10}(x)$ have positive coefficients.
\end{lemma}

\begin{proof} 
First, note
$$P_{12}(x)=x\phi_p^2\phi_{p^2}^2\phi_{pq}\phi_{p^2q}^2=x(\phi_{p}\phi_{pq})(\phi_{p^2}\phi_{p^2q})=x\phi_p{x^q}\phi_{p^2}{x^q}.$$

Next, note
$$Q_{10}(x)=x\phi_q^2\phi_{pq}=x(\phi_{q}\phi_{pq})(\phi_{q})=x\phi_{q}(x^p)\phi_{q}$$
\end{proof}

\begin{lemma}\label{lemma:P_22}
$P_{22}(x)$ and $Q_{00}(x)$ have positive coefficients.
\end{lemma}

\begin{proof} 
First, note
$$P_{22}(x)=x\phi_p^2\phi_{p^2}^2\phi_{pq}^2\phi_{p^2q}^2=x(\phi_{p}\phi_{pq})^2(\phi_{p^2}\phi_{p^2q})^2=x(\phi_p{x^q})^2(\phi_{p^2}(x^q))^2.$$

Next, note $Q_{00}(x)=x\phi_q^2$, which manifestly has positive coefficients in its current form.
\end{proof}

\begin{lemma}\label{lemma:P_21}
$P_{21}(x)$ always has positive coefficients.
\end{lemma}
\begin{proof}
$$P_{21}(x)=x\phi_p^2\phi_{p^2}^2\phi_{pq}^2\phi_{p^2q}=x(\phi_p\phi_{pq})^2(\phi_{p^2}\phi_{p^2q})=x\phi_p^2(x^q)\phi_{p}(x^{pq})$$
\end{proof}

\section{Computational Results}

\begin{proposition}\label{prop:neg_ceil_floor}
  $\lfloor -x\rfloor=-\lceil x\rceil$
\end{proposition}

\begin{proposition}\label{prop:floor2ceil}
  $\lfloor x + \frac{1}{2}\rfloor=\lceil x\rceil $
\end{proposition}

We now mention a few summation identities. The first one is a classical combinatorics result. 

\begin{proposition}\label{prop:triangle}
$$\sum_{k=1}^n k = \frac{n(n+1)}{2}$$
\end{proposition}

We provide proofs for the next two identities.

\begin{proposition}\label{prop:alt_odd}
 $$\sum_{k=0}^n (2k+1)(-1)^k=(-1)^n(n+1).$$
\begin{proof}
We will let $\displaystyle S_n=\sum_{k=0}^n(2k-1)(-1)^k$.
We proceed using induction on $n$. \\This holds for the base case when $n=0$, as
$$(2(0)+1)(-1)^0=1=(-1)^0(0+1).$$
Assume for some $m\geq 0$ we have $S_m=(-1)^m(m+1)$. Now,

\begin{align*}
S_{m+1} &= \sum_{k=0}^{m+1}(2k+1)(-1)^k
\\&=S_m+ (2(m+1)+1)(-1)^{m+1}\\
&= (-1)^m(m+1) + (2m+3)(-1)^{m+1}\\
&= (-1)^{m+1}(-m-1)+(2m+3)(-1)^{m+1}\\
&= (-1)^{m+1}(-m-1+2m+3)\\
&= (-1)^{m+1}(m+2)\\
&= (-1)^{m+1}((m+1)+1)
\end{align*}
as desired. \end{proof}

\end{proposition}

%\cooper{is it satisfactory?}
\begin{proposition}\label{prop:alt_even}
$$\sum_{k=0}^n 2k(-1)^k=(-1)^n2\left\lceil{\frac{n}{2}}\right\rceil.$$
\begin{proof}
We will let $\displaystyle S_n = \sum_{i=0}^n 2i(-1)^i$ and proceed using induction on $n$.

The base case of $n = 0$ yields $$2(0)(-1)^0 = 0= (-1)^0 2\left\lceil \frac{0}{2}\right\rceil $$

Assume for $ m \geq 0$ we have  $S_m = (-1)^m2\lceil{\frac{m}{2}}\rceil$. Observe

\begin{align*}
S_{m+1} &= (-1)^{m+1} 2 \left\lceil \frac{m+1}{2} \right\rceil \\
&=(-1)^m 2\bigl(\left\lceil \tfrac{m}{2}\right\rceil - m - 1\bigr) \\
&= (-1)^m 2\left\lceil \frac{-m-2}{2} \right\rceil \\
&= (-1)^m 2\bigl(-\lfloor \tfrac{m+2}{2}\rfloor\bigr) \\
&= (-1)^m 2\bigl(-\left\lfloor \tfrac{m+2}{2}+\tfrac{1}{2}\right\rfloor\bigr) \\
&= (-1)^{m+1} 2\left\lceil \frac{m+1}{2}\right\rceil
\end{align*}
\end{proof}

\end{proposition}

\begin{lemma}\label{lemma:qpk} Let $q$ and $p$ be integers numbers such that $p<q$. 
There exists an integer $0\leq k\leq p$ so that

$$\left\lfloor\frac{q+k}{p^2}\right\rfloor = \left\lfloor\frac{q+k-p}{p^2}\right\rfloor.$$

\end{lemma}
\begin{proof}
% There is a statement that follows
% $$\lfloor\frac{q+k}{p^2}\rfloor = \lfloor\frac{q+k-p}{p^2}\rfloor$$
% $k$ is an integer where $0\leq k \leq p$.
If the equation holds for when $k=0$, then we are done. If it does not, we will there is an integer $m$ such that
$$q-p < mp^2 \leq q.$$ See Figure \ref{fig:numbline1} for a visual interpretation.

We let $k=mp^2-(q-p)$, thus, $q-p+k=mp^2$.
% value that shifts $q-p$ to equal an $mp^2$. 
% Since the two points maintain a distance $p$, the new point will be $mp^2+p$.
Observe $q+k=mp^2+p$, and thus
\[{q+k\over p^2}={mp^2+p\over p^2}<(m+1).\] 
% Since the stated floor function(s) look at units per $p^2$ for whole numbers, the next 'tick' would be at $(m+1)p^2$; \\

Notably, $q+k<(m+1)p^2$, so there exists a $k$ in the desired range that will make the statement always true. See Figure \ref{fig:numbline2}, again for a visual interpretation.
\end{proof}

%\george{Remove colons from figure}

\begin{figure}[h]
% body of the figure
\begin{center}

\begin{tikzpicture}[x=2.4cm]
\draw[latex-latex] (-2.2,0) -- (2.2,0) ;
\foreach \x/\label in  {-2/{$(m-2)p^2$}, -1/{$(m-1)p^2$}, 0/{$mp^2$}, 1/{$(m+1)p^2$}, 2/{$(m+2)p^2$}} {
  \draw[shift={(\x,0)}] (0pt, 5pt) -- (0pt, -5pt); 
  \node[below] at (\x,0) {\label};
}
\filldraw[black] (-0.3,0) circle (2pt);\filldraw[black] (0.4,0) circle (2pt);
\node[above] at (-0.3,0.1) {$q-p$};
\node[above] at (0.4,0.1) {$q$};
\end{tikzpicture}
\end{center}

\caption{}\label{fig:numbline1}
\end{figure}

\vspace{.2in}
\begin{figure}[h]
% body of the figure

\begin{center}

\begin{tikzpicture}[x=2.4cm]
\draw[latex-latex] (-2.2,0) -- (2.2,0) ; 
\foreach \x/\label in  {-2/{$(m-2)p^2$}, -1/{$(m-1)p^2$}, 0/{$mp^2$}, 1/{$(m+1)p^2$}, 2/{$(m+2)p^2$}} {
  \draw[shift={(\x,0)}] (0pt, 5pt) -- (0pt, -5pt); 
  \node[below] at (\x,0) {\label};
}
\filldraw[black] (0,0) circle (2pt);
\filldraw[black] (0.7,0) circle (2pt);
\node[above] at (0,0.1) {$q-p+k$};
\node[above] at (0.7,0.1) {$q+k$};
\end{tikzpicture}
\end{center}
\caption{}\label{fig:numbline2}
\end{figure}

\begin{cor}\label{cor:qpk} Let $q$ and $p$ be integers numbers such that $2<p<q$. 
There exists an integer $0\leq k\leq p$ so that

$$\left\lfloor\frac{q+k}{p^2}\right\rfloor = \left\lfloor\frac{q+k-p}{p^2}\right\rfloor= \left\lfloor\frac{q+k-2p}{p^2}\right\rfloor.$$
In particular, when $2p>q$, there is $k$ for which all three are $0$.
\end{cor}
\begin{proof}
First, suppose $2p>q$. Note that since $q>p$, $q+p>2p$, and so $p>2p-q$. Thus, we take $k=2p-q$, which is positive as we are assuming $2p>q$. In this case, observe we have 
$$\left\lfloor\frac{q+k}{p^2}\right\rfloor = \left\lfloor\frac{q+k-p}{p^2}\right\rfloor= \left\lfloor\frac{q+k-2p}{p^2}\right\rfloor=0$$
since $p>2$ and thus $2p<p^2$.

Suppose now $q>2p$, and thus, $q-p>p$. By the prior Lemma, replacing $q$ in the statement with $q-p$, we know there is a $k$ for which $\left\lfloor\frac{q+k-p}{p^2}\right\rfloor= \left\lfloor\frac{q+k-2p}{p^2}\right\rfloor$. Due to the proof of the prior Lemma, this $k$ requires $q+k-2p$ to be divisible by $p^2$. Since $2p<p^2$, as $p>2$, this means $q+k-(q+k-2p)<p^2$, and thus $\left\lfloor\frac{q+k-2p}{p^2}\right\rfloor= \left\lfloor\frac{q+k}{p^2}\right\rfloor$.
\end{proof}

\begin{lemma}\label{lem:q01_p2_comp}
Let $\alpha=\lfloor {k\over 2}\rfloor$ and $\beta=\lfloor{k-q\over 2}\rfloor$. Then 

\begin{equation}\label{eq:floorsnonneg}
\delta(\alpha \text{ even}) k -(-1)^{\alpha}\alpha +\delta(\beta \text{ even})(2q-2k)+(-1)^\beta 2\beta -1
\end{equation}

is non-negative for $q<k<2q-2$.
\end{lemma}
\begin{proof}

We proceed by cases on the parity of $\alpha$ and $\beta$. Throughout, we freely use the identity \[\left\lfloor{x\over 2}\right\rfloor=\left\lceil {x+1\over 2}\right\rceil-1\]

\noindent\textbf{Case 1:}
We start with $\alpha$ being even. In this case, while we originally had the assumption $k<2q-2$, we may  $k\leq 2q-4$, as if $k=2q-3$, we have 
\[\alpha=\left\lfloor {2q-3\over 2}\right\rfloor=\lfloor q-{3\over 2}\rfloor=q-2.\].

\noindent \textbf{Case 1(a):}
Now, if  $\beta$ is additionally even, Equation \ref{eq:floorsnonneg} becomes

\begin{align*}
-\alpha+2q-k+2\beta-1&=\left\lceil{-k\over 2}-k \right\rceil+ 2\left\lceil q+{k-q+1\over 2}\right\rceil-3\\
&=\left\lceil{-3k\over 2} \right\rceil+ 2\left\lceil {k+q+1\over 2}\right\rceil-3.
\end{align*}

Observe $-3k$ is even if and only if $k+q+1$ is even, as $q$ is an odd prime. When both are even, both values in the ceiling functions are integers, and we get ${2q+2-k\over 2}-3$, and since we may assume $k\leq 2q-4$, this becomes
\[{2q+2-k\over 2}-3\geq {6\over 2}-3=0\]
as desired. When both are odd, then $k-q$ is even, and we have
\begin{align*}
\left\lceil{-3k\over 2} \right\rceil+ 2\left\lceil {k+q+1\over 2}\right\rceil-3&=
\left\lceil{-3k\over 2} \right\rceil+ k+q+2\left\lceil {1\over 2}\right\rceil-3\\
&=
\left\lceil{2q-k\over 2} \right\rceil-1,
\end{align*}

Which must be nonnegative as $k<2q-2$. 

\noindent\textbf{Case 1(b):} Suppose $\beta$ is odd. Then Equation \ref{eq:floorsnonneg} becomes
\begin{align*}
k-\alpha -2\beta-1&=\left\lceil {k\over 2}\right\rceil+2\left\lceil{q-k\over 2}\right\rceil-1\\
&= \left\lceil {q\over 2}\right\rceil+\left\lceil{q-k\over 2}\right\rceil-1,
\end{align*}
where the last equality holds as $k$ is even if and only if $q-k$ is odd, and so one of ${k\over 2}$ or $q-k\over 2$ is always an integer. As at the end of Case 1(a), we proceed in two ways. Why $k$ is odd, ${q-k\over 2}$ is an integer and the above becomes
\[\left\lfloor{2q-k\over 2}\right\rfloor -1\]
which is nonnegative provided $k\leq 2q-2$. Otherwise, when $k$ is even, the above becomes
\[2\left\lfloor{q\over 2}\right\rfloor -{k\over 2}-1=q-{k\over 2}\geq q-{2q-2\over 2}=1.\]

\noindent \textbf{Case 2:} Now we assume that $\alpha$ is odd. 

\noindent\textbf{Case 2(a):}
When $\beta$ is even, Equation \ref{eq:floorsnonneg} becomes

\begin{align*}
\alpha+2q-2k+2\beta-1
&=\left\lfloor{-3k\over 2}\right\rfloor +2\left\lfloor {k+q\over 2}\right\rfloor-1\\
&=\left\lfloor{2q-k\over 2}\right\rfloor-1
\end{align*}
where the last equality holds since $-3k$ is even if and only if $k+q$ is odd. Since $k<2q-2$, this is always nonnegative. 

\noindent \textbf{Case 2(b):} If $\beta$ is odd, we have 
\begin{align*}
\alpha-2\beta-1&=\alpha+2\left\lceil {q-k\over 2}\right\rceil -1\\
&=\left\lceil{k+1\over 2}\right\rceil+2\left\lceil {q-k\over 2}\right\rceil -2.
\end{align*}
Note $k+1$ is even if and only if $q-k$ is even. When both are even, the fractions in the above ceiling functions are integers, and we get
\[{2q-k+1\over 2}-2\geq {4\over 2}-2=0\]
as $k\leq 2q-3$. If both $k+1$ and $q-k$ are odd, and thus $k$ is even, we have

\begin{align*}
\left\lceil{k+1\over 2}\right\rceil+2\left\lceil {q-k\over 2}\right\rceil -2&=\left\lceil {1\over 2}\right\rceil+\left\lceil {q-k\over 2}\right \rceil+\left\lceil {q\over 2}\right \rceil-2\\
&=\left\lceil {q-k\over 2}\right \rceil+\left\lceil{q+1\over 2}\right\rceil -1\\
&=\left\lceil {2q-k+1\over 2}\right \rceil -1,
\end{align*}
where here we leverage the fact that $q$ is odd. This is nonnegative since $k<2q-2$.

\end{proof}

\bibliographystyle{plain} % We choose the "plain" reference style
\bibliography{ref[1]}

\end{document}